\documentclass[runningheads]{llncs}
\usepackage{graphicx}
\usepackage{mathtools}
\usepackage{amsmath}
\usepackage{amssymb}
\usepackage{scalerel,stackengine}

\begin{document}
\title{A computational proof of the existence of the Dual Isogeny }
\author{Markos Karameris\inst{1}\orcidID{0000-0001-7009-5285} }
\authorrunning{M. Karameris}
\institute{National Kapodistrian University of Athens}
\maketitle              % typeset the header of the contribution
\begin{abstract}
For $E$ an elliptic curve over a perfect field $K$, we present a proof of the existence of the dual isogeny $\hat{\phi}$ using computational methods linked to V{\'e}lu's formulae instead of the standard Galois correspondence method.
\end{abstract}

\section{Preliminaries and Notation}
Let $E(K)$ denote an elliptic curve over a perfect field $K$. 
For simplicicty we assume that $char(K)\ne 2$ or $3$. We will mostly preserve the notation of [\ref{silv}]. \\
To a finite subgroup of $E$ we can associate an isogeny $\phi:E\to E'$. This isogeny is not uniquely determined but if we require the isogeny to be normalized then we do get a unique morphism. 
\begin{definition}
If $\omega'$ is the invariant differential of $E'$ then the pullback $\phi^*(\omega')$ is a holomorphic differential of $E$ and since the space $\Omega_E$ is $1-$dimensional we have $\phi^*(\omega')=c\omega$ for some $c\in K$. We call $\phi$ normalized iff $c=1$. 
\end{definition}
The proof of the existence is presented in [\ref{silv}] in two parts: 
\begin{proposition}
Let $\phi:E\to E_1$ and $\psi:E\to E_2$ be two isogenies with $\phi$ seperable and $ker(\phi)\subset ker(\psi)$. Then there exists a unique isogeny $\lambda:E_1\to E_2$ with $\psi=\lambda\circ\phi$.
\end{proposition} \label{main}
The theorem is then proved in Theorem 6.1 using the above proposition for $\psi=[m]$. Notice that requiring the isogeny to be seperable is natural as an inseperable isogeny cannot be normal. 
The primary tools we will be using are V{\'e}lu's formulae:
Let $G$ be a subgroup of $E(K)$ and write $P=(x_P,y_P)$ then as in [\ref{vel}] the formula we will be using is:
\begin{equation}
    \phi(P)=(x(P)+\sum_{Q\in{G}\ne{O}}x(P+Q)-x(Q),y(P)+\sum_{Q\in{G}\ne{O}}y(P+Q)-y(Q)), P\not\in G
\end{equation}
where we also set $\phi(P)=O, \forall P\in G$ by definition.
\begin{remark}
The above definition is invariant under the action of $G$ and it's kernel is exactly $G$ as intended. It can only be used to describe a normalized isogeny $\phi$ which follows from the uniqueness property.
\end{remark}
\section{Existence of the Dual Isogeny}
The first result we will prove is the same as that of Proposition \ref{main} but in the case of normalized elliptic curve isogenies $\phi,\psi$.
\begin{proposition} \label{org}
Let $\phi:E\to E_1$ and $\psi:E\to E_2$ be two normalized isogenies with $ker(\phi)\subset ker(\psi)$. Then there exists a unique isogeny $\lambda:E_1\to E_2$ with $\psi=\lambda\circ\phi$.
\end{proposition}
\begin{proof}
If $\phi:E\to E_1$ is normalized then we have that $\phi(P)=(\sum_{Q\in{G'}}(x_{P+Q}-x_Q),\sum_{Q\in{G'}}(y_{P+Q}-y_Q))$. But then \\ $\lambda(P)=(\sum_{g\in{G/G'}}(x_{P+\phi(g)}-\phi_x(g)),\sum_{g\in{G/G'}}(y_{P+\phi(g)}-\phi_y(g)))$. 
We observe that: $x_{\phi(P)+\phi(g)}=x_{\phi(P+g)}$ and similarly for $y$ we get: $\lambda\circ\phi(P)=(\sum_{g\in{G/G'}}(\phi_x(P+g)-\phi_x(g)),\sum_{g\in{G/G'}}(\phi_y(P+g)-\phi_y(g)))= \\ (\sum_{g\in{G/G'}}(\sum_{Q\in{G'}}(x_{P+g+Q}-x_{g+Q})),\sum_{g\in{G/G'}}(\sum_{Q\in{G'}}(y_{P+g+Q}-y_{g+Q}))= \\ (\sum_{Q\in{G}}(x_{P+Q}-x_Q),\sum_{Q\in{G}}(y_{P+Q}-y_Q))$
(for a homomorphism $h$ with normal subgroup $G'$, $\sum_{g\in{G}}h(g)=\sum_{g\in{G/G'}}\sum_{Q\in{G}}h(q+Q)$). And thus we obtain a normal isogeny with kernel $G$ which is also unique as it is normalized with $\lambda\circ\phi(P)=\psi(P)$. 
The fact $\lambda$ is an isogeny $E_1\to E_2$ is immediate 
$\lambda(O)=\lambda(\phi(O))=\psi(O)=O$ and
$\lambda(E_2)=\lambda(\phi(E_1))=\psi(E_1)=E_3$ as $\phi,\psi$ are onto.
\end{proof}
So far we restricted our results to normal isogenies however:
\begin{proposition}
For every seperable isogeny $\phi:E\to E'$ there exists an isomorphism $i:E'\to E''$ such that $i\circ\phi$ is normalized
\end{proposition}
\begin{proof}
for a proof see [\ref{lect}].
\end{proof}
We are now ready to prove the original statement of Proposition \ref{main}. for seperable isogenies.
\begin{proof}
Let $i_{\phi},i_{\psi}$ be the corresponding isomorphisms such that $\phi'= i_{\phi}\circ\phi$ and $\psi'= i_{\psi}\circ\psi$ are both normalized. Then by Proposition \ref{org} there is a $\lambda$ such that $\psi'=\lambda\circ\phi'$. Then $i_{\psi}^{-1}\circ\lambda\circ i_{\phi}$ is the required isogeny $E_1\to E_2$.
\end{proof}
\begin{remark}
The fact that we require both $\phi$ and $\psi$ to be seperable does not cause any problems as in the case of inseperable $\psi$ we can simply write $\psi=\pi^n\circ\psi_{sep}$ where $\pi$ is the Frobenius endomorphism (this is possible because we are in a perfect field or even over a finite field of characteristic $p$ essentially). Applying our result to $\phi,\psi_{sep}$ we get an isogeny $\lambda:\lambda\circ\phi=\psi_{sep} \implies \pi^n\circ\lambda\circ\phi=\psi$. The required isogeny is thus $\lambda'=\pi^n\circ\lambda$.
\end{remark}
The above remark settles Proposition \ref{main}. 
\section{Computational Aspects}
The proposed approach also offers a way to compute the dual isogeny directly using V{\'e}lu's formulae. Let $\phi:E\to E/G$ be an isogeny over $K$, then the required steps are:
\begin{enumerate}
    \item compute $\phi_{sep}$ and $n\in\mathbb{N}$ such that $\phi=\phi_{sep}\circ\pi^n$
    \item compute $\hat{\pi}$ via the $[p]$ endomorphism
    \item compute $[m]_{sep}$ and $e\in\mathbb{N}$ such that $[m]=\pi^e\circ [m]_{sep}$
    \item normalize $\phi_{sep}$ using $i_{\phi}$ into $\phi_{norm}$
    \item normalize $[m]_{sep}$ using $i_{[m]}$ into $[m]_{norm}$
    \item using V{\'e}lu's formulae this gives $\lambda:\lambda\circ\phi_{norm}=[m]_{norm}$
    \item set $\widehat{\phi_{sep}}=i_{[m]}^{-1}\circ\pi^e\circ\lambda\circ i_{\phi}$, the dual isogeny is $\widehat{\phi}={\hat{\pi}}^n\circ\widehat{\phi_{sep}}$
\end{enumerate}
Seperating the seperable and inseperable part is a straightforward procedure: writing $\phi$ in terms of rational functions it simply means finding the greatest power of $p$ dividing all the exponents. The second part is also immediate from \ref{silv}, simply write $[p]=[p]_{sep}\circ\pi^k$ and take $\hat{\pi}=[p]_sep\circ\pi^{k-1}$. The only remaining step is the normalization proccess which we will describe explicitly.
\\
\begin{remark}
Given an isogeny $\phi:E\to E'$, we are looking for an isomorphism $i_{\phi}:E'\to E''$ such that $(i_{\phi}\circ\phi)^*(\omega')=\omega$. Writing $\omega=\frac{dx}{2y+a_1x+a_3},\omega'=\frac{dx}{2y+{a_1}'x+{a_3}'}$ and $\phi(x,y)=(r(x),s(x)y+z(x))$ in rational form, abusing notation we get the equation $\frac{d(i_{\phi}(r(x)))}{2i_{\phi}(s(x)y+z(x))+{a_1}'i_{\phi}(r(x))+{a_3}'}=\frac{dx}{2y+a_1x+a_3}$
\end{remark}
We can tackle this problem by noting that $i_{\phi}(x,y)=(u^2x+t_1,u^3y+t_2u^2x+t_3)$ for some $t_i\in K$ which leads us to the following lemma:
\begin{lemma}
Let $i_{\phi}:E\to E'$ be an isomorphism as above, then $i_{\phi}^*(\omega')=\frac{1}{u}\omega$
\end{lemma}
\begin{proof}
The proof is a direct computation along with the observation that the nominator has a term of the form $u^3$ in the denominator and $u^2$ in the nominator.
\end{proof}
Using the Lemma above we immediately obtain the desired result:
\begin{itemize}
    \item for the isogeny $[m]_{sep}$ we simply set $u=(\frac{m}{p^e})^{-1}$ where $p^{e+1}\not|m$
    \item for $\phi_{sep}$ we compute $\phi_{sep}^*(\omega')=c\omega$ and set $u=c^{-1}$
\end{itemize}
Combining these results enables us to compute explicitly the dual isogeny without resorting to the usual method of computing via the preimage of the arbitary isogeny $\phi$.

\end{document}